\documentclass[twoside,a4paper,12pt,reqno]{amsart}
\usepackage{fullpage}
\usepackage{amsfonts}
\usepackage{amssymb}
\usepackage{amsmath}
\usepackage{mathrsfs}
\usepackage{enumitem}
\usepackage{tikz}
\usepackage{float}
\usepackage{subcaption}
\usepackage{cite}
\makeatletter
\newcommand{\citecomment}[2][]{\citen{#2}#1\citevar}
\newcommand{\citeone}[1]{\citecomment{#1}}
\newcommand{\citetwo}[2][]{\citecomment[,~#1]{#2}}
\newcommand{\citevar}{\@ifnextchar\bgroup{;~\citeone}{\@ifnextchar[{;~\citetwo}{]}}}
\newcommand{\citefirst}{\@ifnextchar\bgroup{\citeone}{\@ifnextchar[{\citetwo}{]}}}
\newcommand{\cites}{[\citefirst}
\makeatother
\theoremstyle{plain}
\newtheorem*{claim*}{Claim}
\newtheorem{thm}{Theorem}[section]
\newtheorem{cor}[thm]{Corollary}
\newtheorem{lem}[thm]{Lemma}
\newtheorem{prop}[thm]{Proposition}
\theoremstyle{definition}
\newtheorem{defn}[thm]{Definition}
\newtheorem{ex}[thm]{Example}

\newtheorem{con}[thm]{Construction}

\renewcommand{\l}{\mbox{${\mathcal{L}}$}}
\renewcommand{\r}{\mbox{${\mathcal{R}}$}}
\newcommand{\h}{\mbox{${\mathcal{H}}$}}
\renewcommand{\d}{\mbox{${\mathcal{D}}$}}
\renewcommand{\j}{\mbox{${\mathcal{J}}$}}
\renewcommand{\k}{\mbox{${\mathcal{K}}$}}
\newcommand{\leql}{\leq_{\mathcal{L}}}
\newcommand{\geql}{\geq_{\mathcal{L}}}
\newcommand{\leqr}{\leq_{\mathcal{R}}}
\newcommand{\geqr}{\geq_{\mathcal{R}}}
\newcommand{\leqh}{\leq_{\mathcal{H}}}
\newcommand{\geqh}{\geq_{\mathcal{H}}}
\newcommand{\leqj}{\leq_{\mathcal{J}}}
\newcommand{\geqj}{\geq_{\mathcal{J}}}
\newcommand{\leqk}{\leq_{\mathcal{K}}}

\newcommand{\gl}{>_{\mathcal{L}}}
\newcommand{\lr}{<_{\mathcal{R}}}
\newcommand{\gr}{>_{\mathcal{R}}}
\newcommand{\lh}{<_{\mathcal{H}}}
\newcommand{\gh}{>_{\mathcal{H}}}

\newcommand{\gj}{>_{\mathcal{J}}}
\newcommand{\lk}{<_{\mathcal{K}}}
\newcommand{\gk}{>_{\mathcal{K}}}
\newcommand{\HL}{H_{\mathcal{L}}}
\newcommand{\HR}{H_{\mathcal{R}}}
\newcommand{\HH}{H_{\mathcal{H}}}
\newcommand{\HJ}{H_{\mathcal{J}}}
\newcommand{\HK}{H_{\mathcal{K}}}
\newcommand{\HKS}{H_{\mathcal{K}}^S}
\newcommand{\HRS}{H_{\mathcal{R}}^S}
\newcommand{\HJS}{H_{\mathcal{J}}^S}
\newcommand{\N}{\mathbb{N}}

\newcommand{\U}{\mathcal{U}}

\begin{document}
\title{\large{Heights of posets associated with Green's relations on semigroups}}
\author{Matthew Brookes and Craig Miller}
\address{Department of Mathematics, University of York, UK, YO10 5DD}
\email{craig.miller@york.ac.uk}
\address{School of Mathematics and Statistics, St Andrews, Scotland, UK, KY16 9SS}
\email{mdgkb1@st-andrews.ac.uk}
\maketitle
\begin{abstract}
Given a semigroup $S$, for each Green's relation $\k\in\{\l,\r,\j,\h\}$ on $S,$ the $\k$-height of $S,$ denoted by $\HK(S),$ is the height of the poset of $\k$-classes of $S.$  More precisely, if there is a finite bound on the sizes of chains of $\k$-classes of $S,$ then $\HK(S)$ is defined as the maximum size of such a chain; otherwise, we say that $S$ has infinite $\k$-height.  We discuss the relationships between these four $\k$-heights.  The main results concern the class of stable semigroups, which includes all finite semigroups.  In particular, we prove that a stable semigroup has finite $\l$-height if and only if it has finite $\r$-height if and only if it has finite $\j$-height.  In fact, for a stable semigroup $S,$ if $\HL(S)=n$ then $\HR(S)\leq2^n-1$ and $\HJ(S)\leq2^n-1,$ and we exhibit a family of examples to prove that these bounds are sharp.  Furthermore, we prove that if $2\leq\HL(S)<\infty$ and $2\leq\HR(S)<\infty,$ then $\HJ(S)\leq\HL(S)+\HR(S)-2.$  We also show that for each $n\in\N$ there exists a semigroup $S$ such that $\HL(S)=\HR(S)=2^n+n-3$ and $\HJ(S)=2^{n+1}-4.$  By way of contrast, we prove that for a regular semigroup the $\l$-, $\r$- and $\h$-heights coincide with each other, and are greater or equal to the $\j$-height.  Moreover, in a stable, regular semigroup the $\l$-, $\r$-, $\h$- and $\j$-heights are all equal.
\end{abstract}
~
\maketitle

\section{Introduction}\label{sec:intro}

Green's relations, five equivalence relations based on mutual divisibility, are arguably the most important tools for analysing the structure of semigroups.  For four of these relations, namely $\l,\r,\j$ and $\h$, there is a natural associated poset and thus a height parameter.  Specifically, Green's preorders $\leql,$ $\leqr,$ $\leqj$ and $\leqh$ are defined as follows:
\begin{enumerate}[label=$\bullet$, leftmargin=*]
\item $a\leql b$ if and only if there exists $s\in S^1$ such that $a=sb$;
\item $a\leqr b$ if and only if there exists $s\in S^1$ such that $a=bs$;
\item $a\leqj b$ if and only if there exist $s,t\in S^1$ such that $a=sbt$;
\item $a\leqh b$ if and only if there exist $s,t\in S^1$ such that $a=sb=bt.$
\end{enumerate}
(Throughout this paper, $S^1$ stands for the monoid obtained from $S$ by adjoining an identity $1\notin S.$)
Observe that $\leqh\,=\,\leql\cap\leqr,$ $\leql\,\subseteq\,\leqj$ and $\leqr\,\subseteq\,\leqj$.
Now, letting $\k$ stand for any of $\l,\r,\j$ and $\h$, Green's relation $\k$ is defined by
$$a\,\k\,b\Leftrightarrow a\leqk b\text{ and }b\leqk a.$$
The preorder $\leqk$ induces a partial order on the set of $\k$-classes: 
$$\hspace{12em}K_a\leq K_b\,\Leftrightarrow\,a\leqk b\qquad(\text{where $K_s$ is the $\k$-class of $s\in S$}).$$
\begin{defn}
The {\em $\k$-height} of $S,$ denoted by $\HK(S),$ is defined as follows.  If there is a finite bound on the sizes of chains of $\k$-classes of $S,$ then $\HK(S)$ is defined as the minimum such bound; otherwise, we say that $S$ has {\em infinite $\k$-height} and write $\HK(S)=\infty$.
\end{defn}
The $\l$-, $\r$- and $\j$-heights were first explicitly defined and studied in \cite{Fleischer:2017}, in the context of finite transformation semigroups.  However, these parameters play an implicit role in the Rhodes expansion of a semigroup, a powerful tool in complexity theory, and in variants of this construction; see \cites{Birget}[Chapter XII]{Eilenberg}.  Moreover, for $\k\in\{\l,\r,\j,\h\},$ having finite $\k$-height is clearly stronger than satisfying the minimal condition on $\k$-classes.  The minimal conditions on $\l$-classes, $\r$-classes and $\j$-classes, denoted by $M_L$, $M_R$ and $M_J$, as well as the related conditions $M_L^*$ and $M_R^*$, have played an important role in the development of the structure theory of semigroups; see \cites[Section 6.6]{Clifford:1967}[p.\! 23-29]{Higgins}{Green}{Munn}

The article \cite{Miller} investigates the relationship between the $\r$-heights of semigroups and their bi-ideals (which include left-, right- and two-sided ideals).  In particular, it is shown that the property of having finite $\r$-height is inherited by bi-ideals.  Of course, the results of that paper have obvious duals in terms of $\l$-heights.  In the final section of \cite{Miller} there is a brief discussion about the relationships between the $\l$-, $\r$-, $\j$- and $\h$-heights.  This is the subject of the present article.


In the following section we provide some preliminary definitions and results regarding minimality in the posets of $\l$-, $\r$-, $\j$- and $\h$-classes.  In Section \ref{sec:stability} we consider the notion of stability and its relationship to the $\k$-heights, and in Section \ref{sec:small} we discuss semigroups with $\k$-heights equal to $1$ or $2$.  In order to prove the main results of the paper, in Section \ref{sec:quotients} we develop some machinery concerning quotients and ideal extensions.  The main results, contained in Section \ref{sec:main}, include the following, where $S$ is a semigroup.
\begin{itemize}[leftmargin=*]
\item If $\HL(S)=n<\infty$, then $S$ is stable if and only if $\HR(S)<\infty,$ in which case 
\mbox{$\lceil\log_2(n+1)\rceil\leq\HR(S)\leq2^n-1.$}
\item If $S$ is stable, then $\HL(S)$ is finite if only if $\HJ(S)$ is finite.  Moreover, if $\HL(S)=n<\infty$ then $n\leq\HJ(S)\leq2^n-1.$
\item For every $n\in\N$ and $m\in\{n,\dots,2^n-1\},$ there exists a semigroup $T$ such that $\HL(T)=n$ and $\HR(T)=\HJ(T)=|T|=m.$
\item If $2\leq\HL(S)<\infty$ and $2\leq\HR(S)<\infty$ then, with $\min\big(\HL(S),\HR(S)\big)=n,$ we have
$$\max\big(\HL(S),\HR(S)\big)\leq\HJ(S)\leq\min(2^n-1,\HL(S)+\HR(S)-2).$$
\end{itemize}
Finally, in Section \ref{sec:ss} we consider the relationships between the $\k$-heights for semisimple semigroups and, in particular, for regular semigroups. 

\section{Preliminaries}\label{sec:prelim}

Throughout, $S$ denotes an arbitrary semigroup, and $\k$ stands for any of Green's relations $\l,\r,\j$ and $\h.$  Recall that each of Green's relations is an equivalence relation.  Furthermore, $\l$ is a right congruence, and $\r$ is a left congruence (a {\em left/right congruence} on $S$ is an equivalence relation on $S$ that is preserved under left/right multiplication).

For two elements $a,b\in S,$ we write $a\lk b$ if $a\leqk b$ but $a$ and $b$ are not $\k$-related.
We define a {\em $\k$-chain} in $S$ to be a sequence of elements of $S$ that is strictly decreasing under the preorder $\leqk$.  Note that $S$ has finite $\k$-height if and only if there is a finite bound on the lengths of $\k$-chains in $S.$  Thus, $S$ may have infinite $\k$-height even if all its chains of $\k$-classes are finite.


Notice that a semigroup with finite $\k$-height has minimal and maximal $\k$-classes.  A semigroup can have at most one minimal $\j$-class; if this $\j$-class exists, it is called the {\em minimal ideal}.  On the other hand, a semigroup may possess multiple minimal $\l$/$\r$-classes (also known as minimal left/right ideals).  If a semigroup has minimal $\l$/$\r$-classes, then it has a minimal ideal, which is equal to the union of all the minimal $\l$/$\r$-classes \cite[Theorem 2.1]{Clifford:1948}.

A semigroup is {\em left/right simple} if it has a single $\l/\r$-class, and {\em simple} if it has a single $\j$-class.  Note that a semigroup is simple if and only if it has $\j$-height 1.  Certainly left/right simple semigroups are simple.  It turns out that minimal $\l$/$\r$-classes are left/right simple subsemigroups \cite[Theorem 2.4]{Clifford:1948}, and minimal ideals are simple subsemigroups \cite[Theorem 1.1]{Clifford:1948}.  A {\em completely simple} semigroup is a simple semigroup that possesses both minimal $\l$-classes and minimal $\r$-classes.  

From the preceding discussion, we immediately deduce the following lemma. 

\begin{lem}\label{lem:minideal}
For $\k\in\{\l,\r,\j\}$, if $\HK(S)$ is finite then $S$ has a minimal ideal, which is the union of the minimal $\k$-classes of $S.$
\end{lem}

In fact, the statement of Lemma~\ref{lem:minideal} also holds for $\k=\h$, as a consequence of the following stronger result.

\begin{lem}\label{lem:CSminideal}
A semigroup has minimal $\h$-classes if and only if it has a completely simple minimal ideal (which is the union of all the minimal $\h$-classes).  Consequently, a semigroup with finite $\h$-height has a completely simple minimal ideal.
\end{lem}

\begin{proof}
This lemma is probably well known, but we provide a proof for completeness.

Suppose that $S$ is a semigroup with minimal $\h$-classes, and let $J$ denote the union of all the minimal $\h$-classes.  For any $x\in J$ and $s,t\in S^1$, we have $x\geqh xsxtx,$ so that $x\,\h\,xsxtx$ by minimality, which implies that $x\,\j\,sxt.$  Thus $J$ is the minimal ideal of $S,$ and hence $J$ is simple.  Moreover, for each $x\in J$ we have $x\,\h\,x^2,$ so that, by Green's Theorem \cite[Theorem 2.16]{Clifford:1961}, the $\h$-class of $x$ is a group.  Thus $J$ is a union of groups.  Hence, by \cite[Theorem 4.5]{Clifford:1961}, $J$ is completely simple.

Conversely, if $S$ has a completely simple minimal ideal, say $J,$ then $J$ is the union both of the minimal $\l$-classes and of the minimal $\r$-classes \cite[Theorem~3.2]{Clifford:1948}.  Consequently, the $\h$-classes of $S$ contained in $J$ are minimal. 
\end{proof}

For semigroups with zero, the theory of minimal $\k$-classes becomes trivial, so we require the notion of 0-minimality.  Suppose that $S$ has a zero element 0.  A $\k$-class of $S$ is called {\em 0-minimal} if $\{0\}$ is the only $\k$-class below it.  The semigroup $S$ is {\em left/right 0-simple} if $S^2\neq\{0\}$ and the $\l/\r$-classes of $S$ are $\{0\}$ and $S{\setminus}\{0\}.$  Similarly, $S$ is {\em 0-simple} if $S^2\neq 0$ and the $\j$-classes of $S$ are $\{0\}$ and $S{\setminus}\{0\}.$ 
A {\em completely 0-simple} semigroup is a 0-simple semigroup that possesses both 0-minimal $\l$-classes and 0-minimal $\r$-classes.


\section{$\k$-heights and Stability}\label{sec:stability}

We now introduce a crucial notion for this paper, namely stability.

\begin{defn}
A semigroup $S$ is {\em left stable} if $\leql\cap\,\j=\l.$  Dually, $S$ is {\em right stable} if $\leqr\cap\,\j=\r.$  Finally, $S$ is {\em stable} if it is both left stable and right stable.
\end{defn}


The class of stable semigroups includes all group-bound semigroups (where {\em group-bound} means that every element has some power belonging to a subgroup) \cite[Proposition 7]{East:2020}, and hence all finite semigroups.  The one Green's relation that we have hitherto not defined is $\d=\l\vee\r\,(=\l\circ\r=\r\circ\l)$, for which, in general, there is no natural associated preorder.  However, we have $\d=\j$ in stable semigroups \cite[Corollary 10]{East:2020}.

For a simple semigroup, being completely simple is equivalent to being stable \cite[Proposition 15]{East:2020}.  Consequently, minimal ideals of stable semigroups are completely simple.  This fact, together with Lemma \ref{lem:minideal}, yields an analogue of Lemma~\ref{lem:CSminideal} for stable semigroups.

\begin{lem}\label{lem:stable,minideal}
Let $\k\in\{\l,\r,\j\}.$  A stable semigroup has minimal $\k$-classes if and only if it has a completely simple minimal ideal (which is the union of all the minimal $\k$-classes).  Consequently, a stable semigroup with finite $\k$-height has a completely simple minimal ideal.
\end{lem}

It follows immediately from the definition that $S$ is left stable if and only if every $\j$-class of $S$ is a union of pairwise incomparable $\l$-classes.  In fact, it turns out that $S$ is left stable if and only if it satisfies the condition $M_L^*$, that for each $\j$-class $J$ of $S$ the set of $\l$-classes contained in $J$ has a minimal element \cite[Lemma 6.42]{Clifford:1967}.  Clearly, if $S$ has finite $\l$-height then it satisfies $M_L^*$.  Thus we have:

\begin{lem}\label{lem:stable}
If $\HL(S)$ (resp.\ $\HR(S)$) is finite, then $S$ is left (resp.\ right) stable.  Consequently, if both $\HL(S)$ and $\HR(S)$ are finite, then $S$ is stable.
\end{lem}

It turns out that semigroups with finite $\h$-height are also stable.  In fact, we have the following stronger result.

\begin{lem}\label{lem:gb}
If $\HH(S)=n<\infty,$ then for every $a\in S$ the element $a^n$ belongs to a subgroup of $S.$  In particular, $S$ is group-bound (and hence stable).
\end{lem}

\begin{proof}
Let $a\in S.$  Then 
$$a\geqh a^2\geqh\cdots\geqh a^{n+1}\geqh\cdots.$$
Since $\HH(S)=n,$ there exist $i,j\in\{1,\dots,n+1\}$ with $i<j$ such that $a^i\,\h\,a^j.$  Then, in particular $a^i\,\l\,a^j$ so there is $s\in S^1$ such that $a^i=sa^j$.  It follows by an easy induction argument that $a^n=s^ka^{k(j-i)+n}$ for all $k\in\N.$  Thus, we have $a^n=s^na^{n(j-i+1)}=s^na^{n(j-i-1)}a^{2n}$, and hence $a^n\,\l\,a^{2n}$.  A dual argument proves that $a^n\,\r\,a^{2n}$, and hence $a^n\,\h\,a^{2n}$.  Thus the $\h$-class of $a^n$ is a subgroup of $S,$ as required.
\end{proof}

The next result is the first on the main theme of the article: to explore the relationships between the four $\k$-heights.

\begin{prop}\label{prop:stable}
Let $S$ be a semigroup.  
\begin{enumerate}[leftmargin=*]
\item Suppose that $S$ is left stable.  Then every $\h$-chain of $S$ is an $\r$-chain, and every $\l$-chain of $S$ is a $\j$-chain.  Consequently, $\HH(S)\leq\HR(S)$ and $\HL(S)\leq\HJ(S).$
\item Suppose that $S$ is right stable.  Then every $\h$-chain of $S$ is an $\l$-chain, and every $\r$-chain of $S$ is a $\j$-chain.  Consequently, $\HH(S)\leq\HL(S)$ and $\HR(S)\leq\HJ(S).$
\item Suppose that $S$ is stable.  Then every $\h$-chain of $S$ is both an $\l$-chain and an $\r$-chain, and all $\l$-chains and $\r$-chains of $S$ are $\j$-chains.  Consequently, 
$$\HH(S)\leq\min\big(\HL(S),\HR(S)\big)\quad\text{and}\quad\max\big(\HL(S),\HR(S)\big)\leq\HJ(S).$$
\end{enumerate}
\end{prop}

\begin{proof}
(1) Consider $a,b\in S$ with $a\lh b.$  Then $a\leql b$ and $a\leqr b.$  If $a\,\r\,b$, then certainly $a\,\j\,b$, which implies that $a\,\l\,b,$ since $S$ is left stable.  Then $a\,\h\,b$, a contradiction.  Thus, we must have $a\lr b.$  We conclude that every $\h$-chain is an $\r$-chain, and hence $\HH(S)\leq\HR(S).$

It follows immediately from the definition of left stability that every $\l$-chain is a $\j$-chain, and hence $\HL(S)\leq\HJ(S).$

(2) holds by left-right duality, and (3) follows from (1) and (2).
\end{proof}

\begin{cor}\label{cor:H}
For any semigroup $S$, either $\HH(S)\leq\min\big(\HL(S),\HR(S)\big)$ or $\HH(S)$ is infinite.
\end{cor}

\begin{proof}
Suppose that $\HH(S)$ is finite.  Then $S$ is stable by Lemma \ref{lem:stable}, and hence $\HH(S)\leq\min\big(\HL(S),\HR(S)\big)$ by Proposition \ref{prop:stable}.
\end{proof}

\section{Small $\k$-heights}\label{sec:small}

In this section we consider the relationships between the different $\k$-heights where at least one of them is equal to 1 or 2.

Recall that $\HJ(S)=1$ precisely when $S$ is simple.  The following result provides necessary and sufficient conditions for a semigroup to have $\l$-height 1 or $\r$-height 1 in terms of the $\j$-height and one-sided stability.

\begin{prop}\label{prop: HL=1}
For a semigroup $S,$ the following hold.
\begin{enumerate}[leftmargin=*]
\item $\HL(S)=1$ if and only if $\HJ(S)=1$ and $S$ is left stable.
\item $\HR(S)=1$ if and only if $\HJ(S)=1$ and $S$ is right stable.
\end{enumerate}
\end{prop}

\begin{proof}
Clearly it suffices to prove (1).  If $\HL(S)=1,$ then $S,$ being the union of minimal $\l$-classes, is simple, so $\HJ(S)=1,$ and $S$ is left stable by Lemma \ref{lem:stable}.  The converse follows from Proposition \ref{prop:stable}(1).
\end{proof}

The assumptions of left stability and of right stability are necessary in the respective parts of Proposition \ref{prop: HL=1}.  Indeed, there exist semigroups with $\r$-height 1 (and hence $\j$-height 1) but with infinite $\l$-height (or vice versa).  For example, the Baer-Levi semigroup $\mathcal{BL}_X$ on an infinite set $X,$ defined as the set of all injective maps $\alpha : X\to X$ with $|X\backslash X\alpha|=|X|$, is right simple \cite[Theorem 8.2]{Clifford:1967}, so has $\r$-height 1.  On the other hand, we have $\alpha\leql\beta$ in $\mathcal{BL}_X$ if and only if $X\alpha\subseteq X\beta$ and $|X\beta{\setminus}X\alpha|=|X|$ \cite[Theorem 8]{Pinto}; consequently, $\mathcal{BL}_X$ has infinite $\l$-height.

We now provide several equivalent characterisations for a semigroup to have both $\l$-height 1 and $\r$-height 1.

\begin{prop}\label{prop:HK=1}
For a semigroup $S,$ the following are equivalent:
\begin{enumerate}
\item $\HL(S)=1$ and $\HR(S)=1$;
\item $\HL(S)=1$ and $\HR(S)$ is finite;
\item $\HL(S)=1$ and $S$ is stable;
\item $\HR(S)=1$ and $\HL(S)$ is finite;
\item $\HR(S)=1$ and $S$ is stable;
\item $\HJ(S)=1$ and $S$ is stable;
\item $\HH(S)=1$;
\item $S$ is completely simple.
\end{enumerate}
\end{prop}

\begin{proof}
Let $\k\in\{\l,\r,\j,\h\}.$  If $\HK(S)=1$ then $S,$ being the union of minimal $\k$-classes, is simple.  Also, by Lemmas \ref{lem:stable} and \ref{lem:gb}, each of (1), (2), (4) and (7) implies that $S$ is stable.  Thus, if any of (1)-(7) holds, then $S$ is simple and stable, and hence $S$ is completely simple by \cite[Proposition 15]{East:2020}.

Conversely, if $S$ is completely simple, then $S$ is stable and $\HK(S)=1$ for each $\k\in\{\l,\r,\j,\h\}.$
\end{proof}

An immediate consequence of Proposition~\ref{prop:HK=1} is that for a stable semigroup $S$,
\[
\HL(S)=1\,\Leftrightarrow\,\HR(S)=1\,\Leftrightarrow\,\HJ(S)=1\,\Leftrightarrow\,\HH(S)=1\,\Leftrightarrow\, S\text{ is completely simple}.
\]

We now turn to the situation where one of the $\k$-heights is 2. 

\begin{prop}\label{prop:2,3}
If $S$ is a semigroup such that $\HL(S)=2$ or $\HR(S)=2$, then \mbox{$\HJ(S)\in\{2,3\}.$} 
\end{prop}

\begin{proof}
By symmetry, it suffices to consider the case that $\HL(S)=2.$
By Lemma \ref{lem:minideal}, $S$ has a minimal ideal, say $J,$ which is the union of all the minimal $\l$-classes of $S.$  Since $\HL(S)=2,$ the elements in $S{\setminus}J$ are all maximal under the $\l$-order, and in particular the $\l$-classes contained in $S{\setminus}J$ are incomparable.
 
As $\HL(S)$ is finite, $S$ is left stable by Lemma \ref{lem:stable}, and hence $\HJ(S)\geq2$ by Proposition \ref{prop:stable}(1).  Now suppose for a contradiction that $\HJ(S)\geq 4.$  Then there exists a $\j$-chain $(a,b,c,d)$ in $S$ where $d\in J$ (and hence $a,b,c\in S{\setminus}J$).  Thus, there are $s,t,u,v\in S^1$ such that $b=sat$ and $c=ubv.$ 
We claim that $v\in S$ (i.e.\ that $v\neq 1$).  Indeed, if $v=1,$ then $b\geql c,$ which implies that $b\,\l\,c$ (since the $\l$-classes in $S{\setminus}J$ are incomparable), but then $b\,\j\,c,$ a contradiction.  Now, we have $v,c\in S{\setminus}J$ with $v\geql c$, so $v\,\l\,c$.  Thus, there exists $x\in S^1$ such that $v=cx.$  Therefore, we have
$$c=ubv=ubcx=ububvx=ubusatvx.$$
Thus $busa\geqj c,$ so that $busa\in S{\setminus}J$ and hence $a\,\l\,busa.$  Then $a\,\j\,busa\leqj b,$ so $a\leqj b.$  But this contradicts the assumption that $a\gj b.$  Thus $\HJ(S)\leq 3,$ as required.
\end{proof}

For stable semigroups with $\l$-height 2, we obtain a far stronger statement than that of Proposition \ref{prop:2,3}.

\begin{prop}\label{prop:HK=2}
Let $S$ be a semigroup.  If $S$ is stable and $\HL(S)=2,$ then $\HH(S)=2$ and $\HR(S)=\HJ(S)\in\{2,3\}$.
Moreover, the following are equivalent:
\begin{enumerate}
\item $\HL(S)=\HR(S)=2$;
\item $\HJ(S)=\HH(S)=2$;
\item $\HJ(S)=2$ and $S$ is stable.
\end{enumerate}
\end{prop}

\begin{proof}
Suppose that $S$ is stable with $\HL(S)=2.$  Then, by Propositions \ref{prop:stable}(3), \ref{prop:HK=1} and \ref{prop:2,3}, we have $\HH(S)=2$ and $2\leq\HR(S)\leq\HJ(S)\leq3.$  Suppose for a contradiction that $\HR(S)=2$ and $\HJ(S)=3.$  By Lemma \ref{lem:stable,minideal}, $S$ has a completely simple minimal ideal, say $J.$  Thus, there exists a $\j$-chain $(a,b,c)$ in $S$ where $c\in J.$  Now, there exist $s,t\in S^1$ such that $b=sat.$  Then, as $b\notin J$, we have $sa,at\notin J,$ and then
$$a\geql sa\gl csa\quad\text{and}\quad a\geqr at\gr atc.$$
Since $\HL(S)=\HR(S)=2,$ it follows that $a\,\l\,sa$ and $a\,\r\,at,$ so that there exist $u,v\in S^1$ such that $a=usa=atv.$  But then we have $a=usatv=ubv,$ contradicting the fact that $a\gj b.$  Thus $\HR(S)=\HJ(S).$

We now prove the equivalence of (1), (2) and (3).  If (1) holds, then $S$ is stable by Lemma \ref{lem:stable}, and then (2) follows from the first part of this proposition (just proved). That (2) implies (3) follows immediately from Lemma \ref{lem:gb}, and we have (3) implies (1) by Propositions \ref{prop:stable}(3) and \ref{prop:HK=1}.
\end{proof}

We conclude this section by exhibiting a pair of examples of semigroups to demonstrate that the conditions in Propositions \ref{prop:2,3} and \ref{prop:HK=2} are necessary.

Our first example is a semigroup with $\r$-height 3 and infinite $\j$-height.  This provides a negative answer to \cite[Open Problem 5.1]{Miller}, which asks whether there is a general upper bound for the $\j$-height of a semigroup in terms of its $\r$-height.

\begin{ex}\label{ex:R3 Jinf}
Let $S$ be any right simple semigroup that is not completely simple (such as a Baer-Levi semigroup).  Then $\HR(S)=1$ and $S$ has infinite $\l$-height.  Let ${U=S\cup(S\times S)\cup\{0\},}$ and define a multiplication on $U,$ extending that of $S,$ by
$$(a,b)c=(a,bc),\quad c(a,b)=(ca,b)\quad\text{and}\quad(a,b)(c,d)=(a,b)0=0(a,b)=0^2=0$$
for all $a,b,c,d\in S.$  It is straightforward to show that $U$ is a semigroup under this multiplication.

It is easy to show that the poset the $\r$-classes of $U$ is given as follows: $\{0\}$ is the minimum $\r$-class; $S$ is the maximum $\r$-class; the remaining $\r$-classes are the sets $\{a\}\times S$ ($a\in S$), and these all lie between $\{0\}$ and $S,$ and are pairwise incomparable.  It follows that $\HR(U)=3.$

It is also straightforward to show that the poset the $\j$-classes of $U$ is given as follows: $\{0\}$ is the minimum $\j$-class; $S$ is the maximum $\j$-class; the remaining $\j$-classes are the sets $L_a\times S$ (where $a\in S,$ and $L_a$ denotes the $\l$-class of $a$ in $S$), and $L_a\times S\leq L_b\times S$ if and only if $L_a\leq L_b$.  Thus, the poset of $\j$-classes of $U$ is isomorphic to the poset of $\l$-classes of $S$ with minimum and maximum elements adjoined.  Since $\HL(S)$ is infinite, we conclude that $\HJ(U)$ is infinite.
\end{ex}

Next, we show that it is possible for a (necessarily stable) semigroup to have $\h$-height 2 but infinite $\l$-, $\r$- and $\j$-heights.  In order to construct an example of such a semigroup, we first recall the notion of a Rees quotient.
For an ideal $I$ of $S,$ the {\em Rees quotient of $S$ by $I$}, denoted by $S/I,$ is the semigroup $(S{\setminus}I)\cup\{0\},$ where $0\notin S{\setminus}I,$ with multiplication given by
$$a\cdot b=\begin{cases}
ab&\text{ if }a,b,ab\in S{\setminus}I,\\
0&\text{ otherwise.}
\end{cases}$$


\begin{ex}\label{ex:h=2}
Let $X$ be any infinite set, and let $F$ denote the free semigroup on $X,$ i.e.\ the set of all non-empty words over $X.$  Let $I$ be the set of all words over $X$ in which some letter appears at least twice; that is,
$$I=\{w\in F : w=uxvxz\text{ for some }x\in X\text{ and }u,v,z\in F^1\}.$$
Then $I$ is an ideal of $F.$  Let $S$ be the Rees quotient $F/I.$  We note that $S$ is $\j$-trivial (i.e.\ $\j$ is the equality relation on $S$).  Choose distinct elements $x_i\in X$ ($i\in\N$).  Then $(x_1,x_1x_2,x_1x_2x_3,\dots)$ is both an $\r$-chain and a $\j$-chain of $S,$ so that $\HR(S)$ and $\HJ(S)$ are infinite.  Similarly, $\HL(S)$ is infinite.  We claim that $\HH(S)=2.$  Indeed, let $u,v\in S$ with $u\gh v.$  We need to show that $v=0.$  Now, we have $v=wu=uw'$ for some $w,w'\in S.$  If $w=0$ or $w'=0,$ we are done, so assume that $w,w'\in F{\setminus}I.$  If $|w|\leq|u|,$ then, since $wu=uw'$, we have $u=wz$ for some $z\in F^1$, and hence $v=wwz=0.$  If $|w|>|u|,$ then $w=uz'$ for some $z'\in F$, and hence $v=uz'u=0.$  Thus, in either case, we have $v=0,$ as required.
\end{ex}

\section{Quotients and Ideal Extensions}\label{sec:quotients}

In this section we investigate the relationship between the $\k$-heights of a semigroup $S$ and those of certain quotients of $S.$  Our main purpose here is to establish some critical machinery for proving the main results of the paper, appearing in Section \ref{sec:main}.



\begin{prop}\label{prop:quotient}
Let $S$ be a semigroup and let $\rho$ be a congruence on $S.$
\begin{enumerate}[leftmargin=*]
\item For each $\k\in\{\l,\r,\j\}$ we have $\HK(S)\geq\HK(S/\rho).$
\item For each $\k\in\{\l,\r,\j\},$ if $\leqk\cap\,\rho=\k\cap\rho$ then $\HK(S)=\HK(S/\rho)$.
\item If $S$ is left stable and $\leql\cap\,\rho\subseteq\j,$ then $\HL(S)=\HL(S/\rho).$
\item If $S$ is right stable and $\leqr\cap\,\rho\subseteq\j,$ then $\HR(S)=\HR(S/\rho).$ 
\end{enumerate}
\end{prop}

\begin{proof}
(1) We just consider the case $\k=\j.$  The proofs for $\l$ and $\r$ are similar but slightly more straightforward.

Let $T=S/\rho,$ and consider a $\j$-chain $(b_0,\dots,b_n)$ in $T.$  For $i\in\{1,\dots,n\},$ there exist $u_i,v_i\in T^1$ such that $b_i=u_ib_{i-1}v_i$.  If $u_i\in T,$ choose $s_i\in S$ such that $u_i=[s_i]_{\rho};$ otherwise, let $s_i=1.$  Likewise, if $v_i\in T,$ choose $t_i\in S$ such that $v_i=[t_i]_{\rho};$ otherwise, let $t_i=1.$  Now let $a_0\in S$ be such that $[a_0]_{\rho}=b_0$, and for $i\in\{1,\dots,n\}$ set $a_i=s_i\dots s_1a_0t_1\dots t_i$.  We then have $[a_i]_{\rho}=b_i$, and $a_0\geqj a_1\geqj\cdots\geqj a_n$ in $S$.  We cannot have $a_{i-1}\,\j\,a_i$ (in $S$), for that would imply that $b_{i-1}\,\j\,b_i$ in $T.$  Hence, we have a $\j$-chain $(a_0,\dots,a_n)$ in $S.$  Thus $\HJ(S)\geq\HJ(T).$

(2) Again, we just consider the case $\k=\j$.  By (1) we have $\HJ(S)\geq\HJ(S/\rho),$ so it remains to prove the reverse inequality.  So, consider $a,b\in S$ with $a\gj b.$  Then $[a]_{\rho}\geqj[b]_{\rho}$ in $S/\rho$.  Suppose that $[a]_{\rho}\,\j\,[b]_{\rho}$.  It then follows that $[a]_{\rho}=[sbt]_{\rho}$ for some $s,t\in S^1$.  But then, since $a\gj b\geqj sbt,$ we have $(a,sbt)\in\,\leqj\cap\,\rho$ and $(a,sbt)\not\in\j,$ contradicting the assumption.  Thus $[a]_{\rho}\gj[b]_{\rho}$.  We conclude that $\HJ(S)\leq\HJ(S/\rho),$ and hence $\HJ(S)=\HJ(S/\rho),$ as required.
 
We may now quickly prove (3); the proof of (4) is dual.  Since $\leql\cap\,\rho\subseteq\j,$ we have
$$\leql\cap\,\rho=(\leql\cap\,\rho)\cap\j=(\leql\cap\,\j)\cap\rho=\l\cap\rho,$$
where for the final equality we use the assumption that $S$ is left stable.  Thus, by (2), we have $\HL(S)=\HL(S/\rho).$
\end{proof}

It is possible for the $\h$-height of a semigroup $S$ to be less than the $\h$-height of a quotient $S/\rho$ of $S,$ even when $\leqh\cap\,\rho=\h\cap\rho,$ as demonstrated by the following example.

\begin{ex}\label{ex:h,quotient}
Consider the semigroup $S$ from Example \ref{ex:h=2}.  Recall that $S$ is the Rees quotient of the free semigroup $F$ on an arbitrary infinite set $X$ by the ideal $I$ of all words with multiple appearences of the same letter, and that $\HH(S)=2.$
 
For $w\in S{\setminus}\{0\}$, let $C(w)$ denote the set of elements of $X$ appearing in $w.$  We define a relation $\rho$ on $S$ by $u\,\rho\,v$ if and only if either $u=v=0$ or $u,v\in F\backslash I$ with $C(u)=C(v).$  It is straightforward to prove that $\rho$ is a congruence.  In fact, $\rho$ is the smallest congruence on $S$ such that the resulting quotient is commutative.  Since $u\lh v$ if and only if $u=0,$ and $[0]_{\rho}=\{0\},$ it follows that $\leqh\cap\,\rho=\h\cap\rho.$  It is elementary that, for $u,v\in F\backslash I,$ we have $[u]_{\rho}\geqh[v]_{\rho}$ if and only if $C(u)\subseteq C(v).$  Thus, choosing distinct $x_i\in X$ ($i\in\N$), we have an $\h$-chain $([x_1]_{\rho},[x_1x_2]_{\rho},[x_1x_2x_3]_{\rho},\dots),$ and hence $S/\rho$ has infinite $\h$-height.
\end{ex}


In the remainder of this section we focus on Rees quotients, which were introduced in Section \ref{sec:small}.  Although we defined Rees quotients without reference to congruences, they do in fact arise from congruences.  Specifically, for an ideal $I$ of $S,$ the Rees quotient $S/I$ is isomorphic to the quotient of $S$ by the {\em Rees congruence} $\rho_I=\{(a,a) : a\in S{\setminus}I\}\cup(I\times I).$

\begin{prop}\label{prop:RQ}
Let $S$ be a semigroup and let $\k\in\{\l,\r,\j\}.$  
\begin{enumerate}[leftmargin=*]
\item For any ideal $I$ of $S,$ we have $\HK(S)\geq\HK(S/I).$  
\item If $S$ has minimal $\k$-classes, and $J$ is the minimal ideal of $S,$ then $\HK(S)=\HK(S/J).$
\item If $S$ has a completely simple minimal ideal $J,$ then $\HK(S)=\HK(S/J).$
\end{enumerate}
\end{prop}

\begin{proof}
(1) This is an immediate corollary of Proposition \ref{prop:quotient}(1).

(2) Recalling that $J$ is the union of the minimal $\k$-classes, it is straightforward to show that $\leqk\cap\,\rho_J=\k\cap\,\rho_J$.  Hence, using Proposition \ref{prop:quotient}(2) and the fact that $S/\rho_J\cong S/J,$ we have $\HK(S)=\HK(S/J)$.

(3) This follows from (2), upon recalling that a completely simple minimal ideal is the union of all the minimal $\k$-classes.
\end{proof}

For $\k\in\{\l,\r,\h,\j\},$ and for an ideal $I$ of $S,$ we define $H_{\mathcal{K}}^S(I)$ as follows.  If there is a finite bound on the sizes of chains of $\k$-classes of $S$ contained in $I,$ then $H_{\mathcal{K}}^S(I)$ is the minimum such bound; otherwise $H_{\mathcal{K}}^S(I)=\infty.$  Observing that [$b\in I, a\leqk b\Rightarrow a\in I$], it is clear that any $\k$-chain of $S$ splits into two subchains by restricting to $I$ and to $S{\setminus}I,$ and the latter subchain is a $\k$-chain of $S/I$ to which $0$ can be appended.  It follows that
\begin{equation}\label{eq}
\HK(S)\leq\HKS(I)+\HK(S/I)-1. 
\tag{$\ast$}
\end{equation}

For a semigroup $S$ with zero 0, the {\em left socle} of $S$ is the union of $\{0\}$ and all the 0-minimal $\l$-classes of $S.$  It turns out that the left socle of $S$ is a two-sided ideal of $S$; for a proof of this fact and for more information about the left socle, see \cite[Section 6.3]{Clifford:1967}.

\begin{thm}\label{thm:leftsocle}
Let $S$ be a semigroup with zero, and let $I$ be the left socle of $S.$
\begin{enumerate}[leftmargin=*]
\item $\HL(S)$ is finite if and only if $\HL(S/I)$ is finite, in which case $\HL(S)=\HL(S/I)+1.$
\item If $S$ is right stable, then $\HR(S)$ is finite if and only if $\HR(S/I)$ is finite, in which case $\HR(S)\leq2\HR(S/I)+1.$
\item If $S$ is right stable, then $\HJ(S)$ is finite if and only if $\HJ(S/I)$ is finite, in which case 
$$\HJ(S)\leq\HR(S/I)+\HJ(S/I)+1\leq2\HJ(S/I)+1.$$
\end{enumerate}
\end{thm}

\begin{proof}
We may unambiguously let 0 denote both the zero of $S$ and of $S/I.$

(1) For any $\l$-chain $(a_1,\dots,a_n,0)$ in $S,$ we have $a_i\in S{\setminus}I$ for each \mbox{$i\in\{1,\dots,n-1\}$} (since $a_i\gl a_n\neq0$), and hence there is an $\l$-chain $(a_1,\dots,a_{n-1},0)$ in $S/I.$  Thus \mbox{$\HL(S)\leq\HL(S/I)+1$}. 

Now consider an $\l$-chain $(b_1,\dots,b_n,0)$ in $S/I.$  Then $b_n\notin I,$ so, by definition, there exists some $c\in S{\setminus}\{0\}$ with $b_n\gl c.$  Hence, we have an $\l$-chain $(b_1,\dots,b_n,c,0)$ in $S.$  Thus $\HL(S)\geq\HL(S/I)+1,$ and hence $\HL(S)=\HL(S/I)+1.$

(2) By Proposition \ref{prop:RQ}(1), if $\HR(S)$ is finite then so is $\HR(S/I).$  Suppose then that $\HR(S/I)$ is finite.  We shall prove that $\HRS(I)\leq\HR(S/I)+2.$  Then, using (\ref{eq}), we have
$$\HR(S)\leq(\HR(S/I)+2)+\HR(S/I)-1=2\HR(S/I)+1.$$
Assume for a contradiction that $\HRS(I)>\HR(S/I)+2.$  Then, with {$n=\HR(S/I)+1\,(\geq2),$} there is an $\r$-chain $(a_0,a_1,\dots,a_n,0)$ of $S$ where $a_i\in I$ for all $i\in\{0,\dots,n\}.$
Let $s_i\in S$ be such that $a_i=a_{i-1}s_i$ ($1\leq i\leq n$).  For $i\in\{2,\dots,n\},$ let $t_i=s_2\dots s_i$, so that $a_i=a_1t_i$.  Note that $t_n\neq0$, as $a_n\neq 0$.

Suppose first that $t_n\in I,$ i.e.\ $t_n$ belongs to a 0-minimal $\l$-class of $S.$  Since $a_n=a_1t_n\leql t_n$, it follows that $a_n\,\l\,t_n$.  Thus, there exists $x\in S^1$ such that $t_n=xa_n$, and hence
$$t_n=xa_1t_n=xa_1xa_n=xa_1xa_0s_1t_n.$$
It follows that $a_0\geql a_1xa_0\neq 0$.  Since $a_0\in I,$ we have $a_0\,\l\,a_1xa_0$, which, together with $a_0\gr a_1$, implies that $a_0\,\j\,a_1$.  But this contradicts the fact that $S$ is right stable.  

Now suppose that $t_n\in S{\setminus}I.$  Then we have a chain
$$t_2\geqr t_3\geqr\dots\geqr t_n\gr0$$
in $S/I.$  It cannot be the case that all the inequalities in this chain are strict, for then there would be an $\r$-chain of length $n$ in $S/I,$ contradicting the fact that $\HR(S/I)=n-1.$  Thus, there exists $i\in\{2,\dots,n-1\}$ such that $t_i\,\r\,t_{i+1}$.  But then, using the fact that $\r$ is a left congruence on $S,$ we have $a_i=a_1t_i\,\r\,a_1t_{i+1}=a_{i+1},$ contradicting the fact that $(a_0,\dots,a_n,0)$ is an $\r$-chain.  This completes the proof.

(3) By Proposition \ref{prop:RQ}(1), if $\HJ(S)$ is finite then so is $\HJ(S/I).$  Suppose then that $\HJ(S/I)$ is finite.  By assumption $S$ is right stable, which implies that $S/I$ is right stable.  Therefore, by Proposition \ref{prop:stable}(2), we have $\HR(S/I)\leq\HJ(S/I).$  The second inequality in the statement immediately follows.  To prove the first inequality, we show that $\HJS(I)\leq\HR(S/I)+2,$ and the inequality then follows from (\ref{eq}). 

Consider a $\j$-chain $(a_1,\dots,a_n,0)$ of $S$ where $a_i\in I$ for all $i\in\{1,\dots,n\}.$  We need to show that $n\leq\HR(S/I)+1.$  For each $i\in\{1,\dots,n-1\},$ there exist $s_i,t_i\in S^1$ such that $a_{i+1}=s_ia_it_i$, so $a_{i+1}=s_i\dots s_1a_1t_1\dots t_i$.  For $i\in\{0,\dots,n-1\},$ let $b_{i+1}=s_i\dots s_1a_1$ and $c_{i+1}=a_1t_1\dots t_i$, interpreting $b_1=c_1=a_1$.  Then $b_i,c_i\in I{\setminus}\{0\},$ and we have
$$b_1\geql\cdots\geql b_n\gl0\qquad\text{and}\qquad c_1\geqr\cdots\geqr c_n\gr0.$$
Since the elements of $I{\setminus}\{0\}$ belong to 0-minimal $\l$-classes, it follows that all the $b_i$ are $\l$-related.  Now, aiming for a contradiction, suppose that $c_i\,\r\,c_{i+1}$ for some $i\in\{1,\dots,n-1\}.$  Then there exist $x,y\in S^1$ such that $b_i=xb_{i+1}$ and $c_i=c_{i+1}y$.  But then 
\begin{align*}
a_i&=s_{i-1}\dots s_1a_1t_1\dots t_{i-1}=b_it_1\dots t_{i-1}=xb_{i+1}t_1\dots t_{i-1}=xs_i\dots s_1a_1t_1\dots t_{i-1}\\
&=xs_i\dots s_1c_i=xs_i\dots s_1c_{i+1}y=xs_i\dots s_1a_1t_1\dots t_iy=xa_{i+1}y,
\end{align*}
contradicting the fact that $a_i\gj a_{i+1}$.  We conclude that $(c_1,\dots,c_n,0)$ is an $\r$-chain of $S$ contained in $I.$  We have already established, in the proof of (2), that {$\HRS(I)\leq\HR(S/I)+2,$} so we conclude that $n\leq\HR(S/I)+1,$ as required.
\end{proof}

We note that one can of course dually define the {\em right socle} of a semigroup with zero, and Theorem~\ref{thm:leftsocle} has an obvious counterpart in terms of the right socle.

To conclude this section we shall prove that the upper bounds established in parts (2) and (3) of Theorem \ref{thm:leftsocle} are sharp.  To this end, we introduce a construction in the form of a specific ideal extension of a null semigroup.  (A semigroup $N$ is {\em null} if it contains an element 0 such that $N^2=\{0\}$.)

\begin{con}\label{con:U}
Let $S$ be a semigroup with zero $z.$  Let $\{x_s : s\in S^1\}$ be a set disjoint from $S$ in one-to-one correspondence with $S^1$, and let $\U(S)=S\cup\{x_s : s\in S^1\}.$  Define a multiplication on $\U(S),$ extending that on $S,$ by
$$ax_s=x_s,\quad x_sa=x_{sa}\quad\text{and}\quad x_sx_t=x_z$$
for all $a\in S$ and $s,t\in S^1$.  It is straightforward to show that $\U(S)$ is a semigroup with zero $x_z$.
\end{con}

We focus on this construction for the rest of the section.  We first prove some elementary facts about the structure of $\U(S)$, and then describe its $\l$- and $\r$-heights in terms of those of $S$.

\begin{lem}\label{lem:U}
Let $S$ be a semigroup with zero $z,$ and let $U=\U(S).$ 
\begin{enumerate}[leftmargin=*] 
\item $S$ is a subsemigroup of $U,$ and $\{x_s : s\in S^1\}$ is a null semigroup and an ideal of $U.$
\item The left socle of $U$ is $I=\{x_s : s\in S^1\}\cup\{z\},$ and $U/I\cong S.$
\item If $S$ is right stable then so is $U.$
\end{enumerate}
\end{lem}

\begin{proof}
(1) This follows immediately from the definition of the multiplication in $U.$  

(2) For each $u\in I$ we have $U^1u=\{u,x_z\},$ and hence $I$ is contained in the left socle of $U.$  For $v\notin I,$ we have $v\in S{\setminus}\{z\},$ and hence $v\gl z\gl x_z$ in $U,$ so that $v$ is not in the left socle of $U.$  Thus $I$ is the left socle of $U.$

Letting 0 denote the zero of $U/I,$ it is straightforward to see that there is an isomorphism $U/I\to S$ given by $s\mapsto s$ ($s\in S{\setminus}\{z\}$) and $0\mapsto z.$


(3) It is easy to show that the relations $\r$ and $\j$ on $U$ restricted to $S$ coincide with the corresponding relations on $S.$  Moreover, $\{x_1\}$ is a singleton $\j$-class of $U$, and for $a,b\in S$ we have $x_a\leqj x_b$ if and only if $x_a\leqr x_b$ if and only if $a\leqr b.$  Consequently, if $S$ is right stable then so is $U.$
\end{proof}

\begin{prop}\label{prop:U}
Let $S$ be a semigroup with zero $z,$ and let $U=\U(S).$  Then:
\begin{enumerate}[leftmargin=*]
\item $\HL(U)$ is finite if and only if $\HL(S)$ is finite, in which case $\HL(U)=\HL(S)+1.$
\item $\HR(U)$ is finite if and only if $\HR(S)$ is finite, in which case $\HR(U)=2\HR(S)+1.$
\end{enumerate}
\end{prop}

\begin{proof}
(1) This follows from Theorem \ref{thm:leftsocle}(1) and Lemma \ref{lem:U}(2).

(2) Assume that $\HR(S)=n\in\N.$  Then $S$ is right stable by Lemma \ref{lem:stable}, and hence $U$ is right stable by Lemma \ref{lem:U}(3).  Hence, by Theorem \ref{thm:leftsocle}(2) and Lemma \ref{lem:U}(2), we have $\HR(U)\leq2\HR(S)+1.$

Now, there exists an $\r$-chain $(a_1,\dots,a_{n-1},a_n=z)$ in $S.$  Since $U{\setminus}S$ is an ideal, it follows that $(a_1,\dots,a_{n-1},z)$ is an $\r$-chain of $U,$ and that $z\gr zx_1=x_1$ in $U.$  Write $x_{a_i}=x_{i+1}$ for each $i\in\{1,\dots,n\},$ then $x_1a_1=x_2,$ and clearly $x_1\notin a_1U^1,$ so $x_1\gr x_2$.  For each $i\in\{2,\dots,n\}$ there exists $t_i\in S$ such that $a_{i-1}t_i=a_i$, so $x_it_i=x_{i+1}$ and hence $x_i\geqr x_{i+1}$.  Suppose that $x_i\,\r\,x_{i+1}.$  Then $x_i=x_{i+1}u$ for some $u\in U.$  Since $x_{i+1}x_s=x_z$ for each $s\in S^1$ we have $u\in S.$  But then $a_i=a_{i+1}u,$ contradicting the fact that $a_i\gr a_{i+1}$.  Thus $x_i\gr x_{i+1}$.  We conclude that there is an $\r$-chain
$$(a_1,\dots,a_n,x_1,x_2\dots,x_{n+1})$$
in $U,$ so that $\HR(U)\geq2n+1.$  Thus $\HR(U)=2n+1.$
\end{proof}

We may now quickly deduce that the bounds given in parts (2) and (3) of Theorem \ref{thm:leftsocle} can be attained.  Indeed, take any semigroup $S$ with zero such that $\HR(S)=\HJ(S)<\infty,$ let $U=\U(S),$ and let $I$ be the left socle of $U.$  Then $S\cong U/I$ by Lemma \ref{lem:U}.  Thus, using Propositions \ref{prop:U}(2), \ref{prop:stable}(2) and Theorem \ref{thm:leftsocle}(3), we have
$$2\HJ(U/I)+1=2\HR(U/I)+1=\HR(U)\leq\HJ(U)\leq2\HJ(U/I)+1,$$
implying that $\HR(U)=\HJ(U)=2\HR(U/I)+1=2\HJ(U/I)+1.$   See Figure \ref{figure:U} for an illustration.

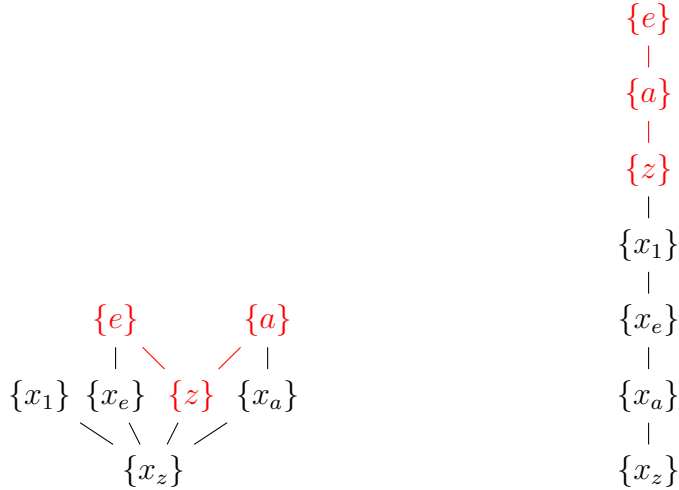
\begin{figure}[!h]
\begin{subfigure}[b]{0.4\textwidth} 
\centering 
{\begin{tikzpicture}[scale=.5]
\node (e) at (-1,4) {$\textcolor{red}{\{e\}}$};
\node (a) at (3,4) {$\textcolor{red}{\{a\}}$};
\node (1) at (-3,2) {$\{x_1\}$};
\node (x) at (-1,2) {$\{x_e\}$};
\node (z) at (1,2) {$\textcolor{red}{\{z\}}$};
\node (y) at (3,2) {$\{x_a\}$};
\node (0) at (0,0) {$\{x_z\}$};
\draw[red] (e) -- (z) -- (a);
\draw (1) -- (0) -- (x) -- (e);
\draw (z) -- (0) -- (y) -- (a);
\end{tikzpicture}}
\end{subfigure}
\begin{subfigure}[b]{0.4\textwidth}
\centering
{\begin{tikzpicture}[scale=.5] 
\node (e) at (8,12) {$\textcolor{red}{\{e\}}$};
\node (a) at (8,10) {$\textcolor{red}{\{a\}}$};
\node (z) at (8,8) {$\textcolor{red}{\{z\}}$};
\node (1) at (8,6) {$\{x_1\}$};
\node (x) at (8,4) {$\{x_e\}$};
\node (y) at (8,2) {$\{x_a\}$};
\node (0) at (8,0) {$\{x_z\}$};
\draw[red] (e) -- (a) -- (z);
\draw (z) -- (1) -- (x) -- (y) -- (0);
\end{tikzpicture}}
\end{subfigure} 
\vspace{-0.5em}
\caption{Let $S=\{e,a,z\}$ be the semigroup with multiplication given by $e^2=e,$ $ea=a$ and $ae=sz=zs=z$ ($s\in S$).  (Then $S\cong\U(\{e\}).$)  The poset of $\l$-classes of $\U(S)$ is displayed on the left. The poset of $\r$-classes of $\U(S)$ coincides with the poset of $\j$-classes, and is displayed on the right.  These contain, respectively, the posets of $\l$-classes of $S$ and of $\r$($=\!\j$)-classes of $S,$ which are displayed in red.}
\label{figure:U}
\end{figure}


\section{General Bounds on $\k$-heights}\label{sec:main}

This section contains the main results of the article. 
The first main result establishes upper and lower bounds on the $\r$-height of a stable semigroup with finite $\l$-height.

\begin{thm}\label{thm:l,r}
Let $S$ be a semigroup.  If $\HL(S)=n<\infty$, then $S$ is stable if and only if $\HR(S)<\infty,$ in which case 
$$\lceil\log_2(n+1)\rceil\leq\HR(S)\leq2^n-1.$$
\end{thm}

\begin{proof}
That finite $\HR(S)$ (together with finite $\HL(S)$) implies stability is precisely the statement of Lemma \ref{lem:stable}. Therefore it suffices to prove that for a stable semigroup the bounds stated for $\HR(S)$ hold. We prove that $\HR(S)\leq2^n-1$ by induction.  A dual argument then proves that $n=\HL(S)\leq2^m-1,$ where $m=\HR(S),$ which yields the lower bound in the statement.

Suppose then that $S$ is stable. If $n=1$, then $\HR(S)=1\,(=2^1-1)$ by Proposition \ref{prop:HK=1}. Now take $n>1.$  By Lemma \ref{lem:stable,minideal}, $S$ has a completely simple minimal ideal, say $J.$  By Proposition \ref{prop:RQ}(3), we may assume that $S=S/J$; that is, $S$ has a zero.  Let $I$ denote the left socle of $S.$  Then $\HL(S/I)=n-1$ by Theorem \ref{thm:leftsocle}(1).  By the inductive hypothesis we have $\HR(S/I)\leq 2^{n-1}-1.$  Then, using Theorem \ref{thm:leftsocle}(2), we have 
$$\HR(S)\leq 2\HR(S/I)+1\leq 2(2^{n-1}-1)+1=2^n-1,$$
completing the proof of the inductive step and hence of the theorem.
\end{proof}

Recall that it is possible for a semigroup to have finite $\l$-height but infinite $\j$-height (e.g.\ the dual of Example~\ref{ex:R3 Jinf}).  However, a stable semigroup with finite $\l$-height {\em does} have finite $\j$-height.  In fact, in this case the $\j$-height has the same upper bound as that for the $\r$-height (given in Theorem~\ref{thm:l,r}).

\begin{thm}\label{thm:l,j}
Let $S$ be a stable semigroup.  Then $\HL(S)$ is finite if only if $\HJ(S)$ is finite.  Moreover, if $\HL(S)=n<\infty$ then $$n\leq\HJ(S)\leq2^n-1.$$
\end{thm}

\begin{proof}
By Proposition \ref{prop:stable}, we have $\HL(S)\leq\HJ(S).$  So, it suffices to prove the upper bound in the statement.  The proof of this is essentially the same as that of Theorem \ref{thm:l,r}: simply replace $\r$ with $\j,$ and invoke Theorem \ref{thm:leftsocle}(3) rather than Theorem \ref{thm:leftsocle}(2).
\end{proof}

Using the construction from Section \ref{sec:quotients}, we now show that all the possible $\r$- and $\j$-heights according to Theorems \ref{thm:l,r} and \ref{thm:l,j} are in fact attainable.

\begin{thm}\label{thm:n,m}
Let $n\in\N.$  For every $m\in\{n,\dots,2^n-1\},$ there exists a (necessarily $\j$-trivial) semigroup $S$ such that $\HL(S)=n$ and $\HR(S)=\HJ(S)=|S|=m.$
\end{thm}

\begin{proof}
We prove the result by induction.  Note that a finite semigroup $S$ is $\j$-trivial if and only if $\HJ(S)=|S|.$  For $n=1$ (in which case $n=2^n-1$), we take $S$ to be the trivial semigroup, which clearly has the desired properties.  

Now let $n\geq2.$  First consider $m\in\{n,\dots,2^{n-1}\}.$  By the inductive hypothesis, there exists a semigroup $S$ such that $\HL(S)=n-1$ and $\HR(S)=\HJ(S)=|S|=m-1.$  Then $\HL(S^1)=n$ and $\HR(S^1)=\HJ(S^1)=|S^1|=m,$ as required. 

Now consider $m\in\{2^{n-1}+1,\dots,2^n-1\}.$  By the inductive hypothesis, there exists a semigroup $S$ such that $\HL(S)=n-1$ and $\HR(S)=\HJ(S)=|S|=2^{n-1}-1.$  Being finite and $\j$-trivial, the semigroup $S$ has a zero, say $z.$  Let $U=\U(S).$  Clearly $|U|=2|S|+1=2^n-1.$  By Proposition \ref{prop:U}, we have $\HL(U)=n$ and $\HR(U)=\HJ(U)=2^n-1.$  In fact, following the proof of Proposition \ref{prop:U}, and using the fact that $\HR(U)=|U|,$ we may write $U=\{a_1,\dots,a_{2^n-1}\},$
where $S=\{a_1,\dots,a_{2^{n-1}-1}(=z)\},$ such that $(a_1,\dots,a_{2^n-1})$
is an $\r(=\j)$-chain of $U.$  By Lemma \ref{lem:U}, the left socle of $U$ is 
$$(U{\setminus}S)\cup\{z\}=\{a_{2^{n-1}-1},\dots,a_{2^n-1}\}.$$
Let $I=\{a_m,\dots,a_{2^n-1}\}.$  Then $I$ is an ideal of $U$ (since it consists of all elements less than or equal to $a_m$ under the $\j$-order).  Let $T$ denote the Rees quotient 
$$U/I=\{a_1,\dots,a_{m-1},0\},$$
and observe that $S\subseteq T.$  
It is immediate that for $\k\in\{\l,\r,\j\}$ and $x,y\in T{\setminus}\{0\},$ we have $x\gk y$ in $T$ if and only if $x\gk y$ in $U.$  It follows that $\HL(T)=n$ and $\HR(T)=\HJ(T)=|T|=m.$  This completes the proof.
\end{proof}

\begin{cor}
Let $n\in\N.$  For every $m\in\{\lceil\log_2(n+1)\rceil,\dots,2^n-1\},$ there exists a semigroup $S$ such that $\HL(S)=n,$ $\HR(S)=m$ and $\HJ(S)=|S|=\max(m,n).$ 
\end{cor}

\begin{proof}
If $m\geq n$ then we apply Theorem \ref{thm:n,m}.  If $m<n\,(\leq2^m-1),$ then we apply the dual of Theorem \ref{thm:n,m} to obtain a semigroup $S$ such that $\HR(S)=m$ and $\HL(S)=\HJ(S)=|S|=n.$
\end{proof}



\begin{table}[h]
\centering
\begin{tabular}{|c||c||c|}
\hline
$\HL(S)$ & Possible values for $m=\HR(S)$ & Possible values for $p=\HJ(S)$\\
\hline\hline
1 & $m=1$ & $p=1$\\
2 & $m=2,3$ & $p=2,3$\\
3 & $2\leq m\leq 7$ & $3\leq p\leq 7$\\
4 & $3\leq m\leq 15$ & $4\leq p\leq 15$\\
5 & $3\leq m\leq 31$ & $5\leq p\leq 31$\\
6 & $3\leq m\leq 63$ & $6\leq p\leq 63$\\
7 & $3\leq m\leq 127$ & $7\leq p\leq 127$\\
8 & $4\leq m\leq 255$ & $8\leq p\leq 255$\\
\hline
\end{tabular}
\vspace{-0.5em}
\caption{For some small natural numbers $n,$ the range of possible values of $\HR(S)$ and of $\HJ(S)$ for a stable semigroup $S$ with $\HL(S)=n$.}
\label{table:stable}
\end{table}

Recall Proposition \ref{prop:HK=1}, which provides several equivalent conditions for a semigroup $S$ to have $\HL(S)=\HR(S)=1,$ and Proposition \ref{prop:HK=2}, which gives equivalent conditions for $S$ to have $\HL(S)=\HR(S)=2$.  Our next result is an analogue, describing when the $\l$- and $\r$-heights of a semigroup are both finite.

\begin{thm}\label{thm:l,r,h,j}
For a semigroup $S,$ the following are equivalent:
\begin{enumerate}
\item $\HL(S)$ and $\HR(S)$ are finite;
\item $\HL(S)$ and $\HH(S)$ are finite;
\item $\HL(S)$ is finite and $S$ is stable;
\item $\HJ(S)$ and $\HH(S)$ are finite; 
\item $\HJ(S)$ is finite and $S$ is stable;
\item $\HR(S)$ and $\HH(S)$ are finite;
\item $\HR(S)$ is finite and $S$ is stable.
\end{enumerate}
\end{thm}

\begin{proof}
If (1) holds, then $S$ is stable by Lemma \ref{lem:stable}, and hence $\HH(S)$ is finite by Proposition \ref{prop:stable}(3), so (2) holds.  Each of (2)$\Rightarrow$(3), (4)$\Rightarrow$(5) and (6)$\Rightarrow$(7) follows from Lemma \ref{lem:gb}.  That (3)$\Rightarrow$(4) follows from Theorem \ref{thm:l,j} and Proposition \ref{prop:stable}(3), and that (5)$\Rightarrow$(6) follows from Proposition \ref{prop:stable}(3).  Finally, (7)$\Rightarrow$(1) follows from Proposition \ref{prop:stable}(3) and the dual of Theorem \ref{thm:l,r}.
\end{proof}

An important aspect of Theorem~\ref{thm:l,r,h,j} is:

\begin{cor}
If $S$ is a stable semigroup, then 
$$\HL(S)\text{ is finite}\;\Leftrightarrow\;\HR(S)\text{ is finite}\;\Leftrightarrow\;\HJ(S)\text{ is finite},$$
in which case $\HH(S)$ is finite.
\end{cor}

Recall that for a stable semigroup $S$ we have $\HL(S)=1$ if and only if $\HR(S)=1$ if and only if $\HJ(S)=1.$  The next result establishes lower and upper bounds on the $\j$-height in terms of finite (but greater than 1) $\l$- and $\r$-heights.

\begin{thm}\label{thm:l,r,j}
Let $S$ be a semigroup such that $2\leq\HL(S)<\infty$ and $2\leq\HR(S)<\infty.$  Then, letting $\min\big(\HL(S),\HR(S)\big)=n,$ we have
$$\max\big(\HL(S),\HR(S)\big)\leq\HJ(S)\leq\min(2^n-1,\HL(S)+\HR(S)-2).$$
\end{thm}

\begin{proof}
By Lemma \ref{lem:stable}, the semigroup $S$ is stable.  Thus, by Proposition \ref{prop:stable}(3), the lower bound in the statement holds.  The inequality $\HJ(S)\leq2^n-1$ follows from Theorem \ref{thm:l,j} and its dual (which bounds the $\j$-height in terms of the $\r$-height).  It remains to show that $\HJ(S)\leq\HL(S)+\HR(S)-2.$

By Lemma \ref{lem:stable,minideal}, $S$ has a completely simple minimal ideal, which, by Proposition \ref{prop:RQ}(3), we may assume is trivial, i.e.\ $S$ has a zero 0.  
Consider a $\j$-chain $(a_1,\dots,a_n,0)$ in $S.$  We need to show that $n\leq\HL(S)+\HR(S)-3.$ For each $i\in\{1,\dots,n-1\},$ there exist $s_i,t_i\in S^1$ such that $a_{i+1}=s_ia_it_i$, so that $a_{i+1}=s_i\dots s_1a_1t_1\dots t_i$.  For $i\in\{0,\dots,n-1\},$ let $b_{i+1}=s_i\dots s_1a_1$ and $c_{i+1}=a_1t_1\dots t_i$, interpreting $b_1=c_1=a_1$.  We then have chains
$$b_1\geql\cdots\geql b_n\gl0\qquad\text{and}\qquad c_1\geqr\cdots\geqr c_n\gr0.$$
Now, for each $i\in\{1,\dots,n-1\},$ either $b_i\gl b_{i+1}$ or $c_i\gr c_{i+1}$.  Indeed, if we had $b_i\,\l\,b_{i+1}$ and $c_i\,\r\,c_{i+1}$ for some $i\in\{1,\dots,n-1\},$ then, as in the proof of Theorem \ref{thm:leftsocle}(3), we would have $a_i\,\j\,a_{i+1}$, a contradiction.  It follows that there exist $k,l\leq n$ such that $n\leq k+l-1,$ there is an $\l$-chain $(b_{i_1},\dots,b_{i_k},0)$ where $1=i_1<\dots<i_k\leq n,$ and there is an $\r$-chain $(c_{j_1},\dots,c_{j_l},0)$ where $1=j_1<\dots<j_l\leq n.$  We must then have $k\leq\HL(S)-1$ and $l\leq\HR(S)-1,$ whence $$n\leq k+l-1\leq\HL(S)+\HR(S)-3,$$ as required.
\end{proof}

It is an open question as to whether the upper bound on $\HJ(S)$ given in Theorem \ref{thm:l,r,j} can be attained for every possible combination of $\HL(S)$ and $\HR(S)$ (as determined by Theorem \ref{thm:l,r} and its dual).  However, it is asymptotically sharp:

\begin{thm}\label{thm:asym}
For each $n\in\N$ there exists a semigroup $U_n$ such that $\HL(U_n)=\HR(U_n)=2^n+n-3$ and $\HJ(U_n)=2^{n+1}-4.$  Moreover, we have
$$\lim_{n\to\infty}\frac{\HJ(U_n)}{\HL(U_n)+\HR(U_n)-2}=1.$$
\end{thm}

\begin{proof}
Recall that a {\em left identity} of a semigroup $S$ is an element $e\in S$ such that $es=s$ for all $s\in S.$  Right identities are defined dually.  Recalling Construction \ref{con:U}, observe that if $e$ is a left identity of $S$ then it is also a left identity of $\U(S).$

Let $n\in\N.$  By the above observation and the proof of Theorem \ref{thm:n,m}, there exists a semigroup $S$ with a left identity $e$ such that $\HL(S)=n$ and $\HR(S)=\HJ(S)=|S|=2^n-1.$  Dually, there exists a semigroup $T$ with a right identity $f$ such that $\HR(T)=n$ and $\HL(T)=\HJ(T)=|T|=2^n-1.$  Let $0_S$ and $0_T$ denote the zeros of $S$ and $T,$ respectively, and let $I=(S\times\{0_T\})\cup(\{0_S\}\times T).$  Clearly $I$ is an ideal of the direct product $S\times T.$  Now let $U\,(=U_n)=(S\times T)/I.$  

Consider a $\k$-chain $(u_1,\dots,u_m,0)$ in $U,$ where $u_i=(x_i,y_i).$  Then, for each $i\in\{1,\dots,m-1\},$ either $x_i\gk x_{i+1}$ in $S$ or $y_i\gk y_{i+1}$ in $T.$  It follows that there exist $k,l\leq m$ such that $m\leq k+l-1,$ there is an $\k$-chain $(x_{i_1},\dots,x_{i_k},0)$ where $1=i_1<\dots<i_k\leq m,$ and there is an $\k$-chain $(y_{j_1},\dots,y_{j_l},0)$ where $1=j_1<\dots<j_l\leq m.$  We must then have $k\leq\HK(S)-1$ and $l\leq\HK(T)-1,$ whence $m\leq k+l-1\leq\HK(S)+\HK(T)-3.$  It follows that $\HK(U)\leq\HK(S)+\HK(T)-2.$  Thus, we have $\HL(U)\leq2^n+n-3,$ $\HR(U)\leq2^n+n-3$ and $\HJ(U)\leq2^{n+1}-4.$

Note that, for any $\k\in\{\l,\r,\j\}$ and $(a,b),(c,d)\in U{\setminus}\{0\},$ if $(a,b)\,\k\,(c,d)$ then $a\,\k\,c$ in $S$ and $b\,\k\,d$ in $T.$

Now, let $(a_1,\dots,a_{n-1},0_S)$ be a maximal $\l$-chain in $S,$ and let $(b_1,\dots,b_{2^n-2},0_T)$ be {\em the} maximum $\l$-chain in $T.$  Observe that $b_1=f.$ 
For $i\in\{1,\dots,n-2\}$ let $s_i\in S$ be such that $a_{i+1}=s_ia_i$, and for $j\in\{1,\dots,2^n-3\}$ let $t_j\in T$ be such that $b_{j+1}=t_jb_j$.  Then, for each such $i$ and $j,$ we have $(a_{i+1},f)=(s_i,f)(a_i,f)$ and $(a_{n-1},b_{j+1})=(e,t_j)(a_{n-1},b_j).$  It follows that 
$$(a_1,f)\gl\cdots\gl(a_{n-1},f)=(a_{n-1},b_1)\gl\cdots\gl(a_{n-1},b_{2^n-2})\gl0,$$
so that $\HL(U)\geq2^n+n-3.$  Thus $\HL(U)=2^n+n-3.$  Similarly, we have $\HR(U)=2^n+n-3.$

Now let $(c_1,\dots,c_{2^n-2},0_S)$ be the maximum $\r$-chain in $S.$  Letting $s_i'\in S$ ($1\leq i\leq 2^n-3$) be such that $c_{i+1}=c_is_i'$, we have $(c_{i+1},f)=(c_i,f)(s_i',f).$  Also, with $b_j$ and $t_j$ as above, we have $(c_{2^n-2},b_{j+1})=(e,t_j)(c_{2^n-2},b_j).$  It follows that
$$(c_1,f)\gj\cdots\gj(c_{2^n-2},f)=(c_{2^n-2},b_1)\gj\cdots\gj(c_{2^n-2},b_{2^n-2})\gj0,$$
so that $\HJ(U)\geq2(2^n-2)=2^{n+1}-4.$  Thus $\HJ(U)=2^{n+1}-4.$  This completes the proof of the first part of the statement.  The second part now follows immediately.
\end{proof}

An immediate consequence of Theorem \ref{thm:asym} is that the upper bound of Theorem \ref{thm:l,r,j} is sharp in the case that the $\l$- and $\r$-heights are both 3: 
$\HL(U_2)=\HR(U_2)=3$ and $\HJ(U_2)=4\,(=3+3-2).$  In the proof of Theorem \ref{thm:asym} for the case $n=2,$ the semigroup $S$ must be isomorphic to $\U(\{e\}),$ and $T$ must be anti-isomorphic to $\U(\{e\}).$  Writing $S=\{e,x,0_S\}$ and $T=\{f,y,0_T\},$ and letting $a=(e,f),$ $b=(e,y),$ $c=(x,f)$ and $d=(x,y),$ we have $U_2=\{a,b,c,d,0\},$ and the posets of $\l$-, $\r$-, $\j$- and $\h$-classes of $U_2$ are as displayed in Figure \ref{fig1}.  We note that, since $\HH(U_2)=2,$ both the inequalities in Proposition \ref{prop:stable}(3) can be strict.  


\vspace{-0.5em}
\begin{center}
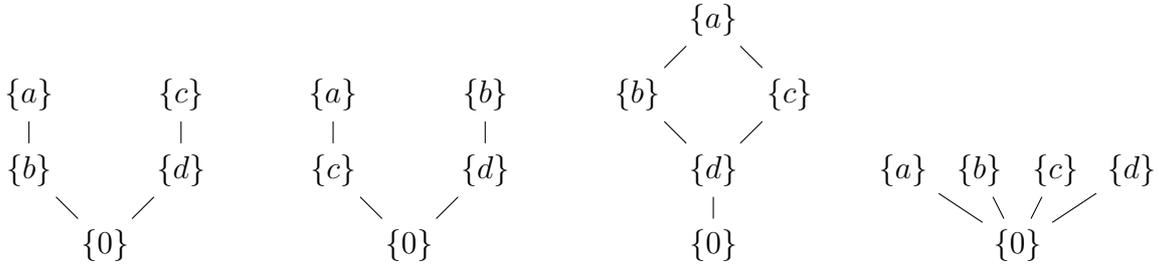
\begin{figure}[!h]
\begin{tikzpicture}[scale=0.5]
\node (a) at (-2,4) {$\{a\}$};
\node (c) at (2,4) {$\{c\}$};
\node (b) at (-2,2) {$\{b\}$};
\node (d) at (2,2) {$\{d\}$};
\node (zero) at (0,0) {$\{0\}$};
\draw (a) -- (b) -- (zero) -- (d) -- (c);
\node (a) at (6,4) {$\{a\}$};
\node (b) at (10,4) {$\{b\}$};
\node (c) at (6,2) {$\{c\}$};
\node (d) at (10,2) {$\{d\}$};
\node (zero) at (8,0) {$\{0\}$};
\draw (a) -- (c) -- (zero) -- (d) -- (b);
\node (a) at (16,6) {$\{a\}$};
\node (b) at (14,4) {$\{b\}$};
\node (c) at (18,4) {$\{c\}$};
\node (d) at (16,2) {$\{d\}$};
\node (zero) at (16,0) {$\{0\}$};
\draw (zero) -- (d) -- (b) -- (a) -- (c) -- (d);
\node (a) at (21,2) {$\{a\}$};
\node (b) at (23,2) {$\{b\}$};
\node (c) at (25,2) {$\{c\}$};
\node (d) at (27,2) {$\{d\}$};
\node (zero) at (24,0) {$\{0\}$};
\draw (a) -- (zero) -- (b) 
(c) -- (zero) -- (d);
\end{tikzpicture}
\vspace{-0.5em}
\caption{For the semigroup $U_2$, the posets of $\l$-classes (left), $\r$-classes (middle left), $\j$-classes (middle right) and $\h$-classes (right).}
\label{fig1}
\end{figure}
\vspace{-1.5em}
\end{center}

%
%

\section{Semisimple and Regular Semigroups}\label{sec:ss}

A crucial reason for the disparities between certain $\k$-heights of the semigroups of Example \ref{ex:h=2} and Theorems \ref{thm:n,m} and \ref{thm:asym} is the presence of $\j$-classes $J$ such that $J^2\cap J=\emptyset.$  In this section we shall see that for stable semigroups without such $\j$-classes all the four $\k$-heights coincide.

The {\em principal factors} of (the $\j$-classes of) a semigroup $S$ are defined as follows.  If $S$ has a minimal $\j$-class (i.e.\ a minimal ideal), then this $\j$-class is defined to be its own principal factor.  For any non-minimal $\j$-class $J$ of $S,$ the principal factor of $J$ is the Rees quotient of the subsemigroup $S^1JS^1$ of $S$ by its ideal $S^1JS^1{\setminus}J.$  Observe that this principal factor has universe $J\cup\{0\},$ where $0$ denotes the zero element.

Recall that the minimal ideal of a semigroup is simple.  All other principal factors are either $0$-simple or null \cite[Lemma 2.39]{Clifford:1961}.  A semigroup is {\em semisimple} if all its principal factors are either simple or 0-simple.

\begin{prop}\label{prop:semisimple}
Let $S$ be a semisimple semigroup.  Then
$$\HJ(S)\leq\min\big(\HL(S),\HR(S)\big).$$
\end{prop}

\begin{proof}
Consider a $\j$-chain $(a_0,a_1,\dots,a_n)$ in $S.$  Let $J_i$ denote the $\j$-class of $a_i$ ($i\in\{0,\dots,n\}$).  Since $J_i=J_i^2=J_i^3$, and $J_i\subseteq S^1a_{i-1}S^1$, there exist $x_i,y_i\in J_i$ and $s_i,t_i\in S^1$ such that $a_i=(x_is_i)a_{i-1}(t_iy_i).$  Put $u_i=x_is_i$ and $v_i=t_iy_i$.  We certainly have $u_i\dots u_1a_0\geql u_{i+1}\dots u_1a_0$ for each $i\in\{1,\dots,n-1\}.$  Also, for each $i\in\{1,\dots,n\},$ we have
$$x_i\geqj u_i\dots u_1a_0\geqj u_i\dots u_1a_0v_1\dots v_i=a_i,$$
implying that $u_i\dots u_1a_0\in J_i$ (since $x_i,a_i\in J_i$).  It follows that
$$a_0\gl u_1a_0\gl u_2u_1a_0\gl\cdots\gl u_n\dots u_1a_0,$$
which is an $\l$-chain of length $n+1$.
We conclude that $\HL(S)\geq\HJ(S).$  A similar argument using the $v_i$ proves that $\HR(S)\geq\HJ(S).$
\end{proof}

Given a semigroup $S,$ there is a natural partial order on the set $E=E(S)$ of idempotents of $S$ given by $e\geq f$ if and only if $ef=fe=f.$
We denote the height of the poset $(E,\leq)$ by $H_E(S).$  

It is straightforward to show that, for $e,f\in E,$ we have
$$e\geql f\Leftrightarrow fe=f\qquad\text{and}\qquad e\geqr f\Leftrightarrow ef=f,$$ from which it follows that [$e\geq f\Leftrightarrow e\geqh f$] and [$e>f\Leftrightarrow e\gl f\text{ and }e\gr f$].  
Moreover, we have [$e\gl f\Leftrightarrow e>ef$ and $e\geql f$].  Indeed, suppose that $e\gl f.$  Then certainly $e\geql f,$ and we have $(ef)e=e(fe)=ef=e(ef),$ so $e\geq ef.$  Since $e\not\leql f,$ it follows that $e\neq ef$ and hence $e>ef.$  Conversely, if $e>ef$ and $e\geql f,$ then we cannot have $e\,\l\,f,$ for that would imply that $e=ef,$ so we must have $e\gl f.$

From the above discussion we deduce:



\begin{lem}\label{lem:E}
Let $S$ be a semigroup, let $E$ denote the set of idempotents of $S,$ and let $n\in\N.$  Then the following are equivalent:
\begin{enumerate}
\item there exists an $\l$-chain of idempotents of length $n$ in $S$;
\item there exists an $\r$-chain of idempotents of length $n$ in $S$;
\item there exists an $\h$-chain of idempotents of length $n$ in $S$;
\item there exists a chain of idempotents of length $n$ in $(E,\leq).$
\end{enumerate}
Consequently, we have $H_E(S)\leq\min(\HL(S),\HR(S),\HH(S)).$
\end{lem}


%
%

A semigroup $S$ is said to be {\em regular} if for every $a\in S$ there exists some $b\in S$ such that $a=aba$ and $b=bab$; the element $b$ is called an {\em inverse} of $a.$  A semigroup is {\em inverse} if each of its elements has a unique inverse. 

A semigroup $S$ is regular (resp.\ inverse) if and only if every $\l$-class and every $\r$-class of $S$ contain at least (resp.\ exactly) one idempotent \cite[Theorem~6]{Green}.  It follows that for any regular semigroup $S$ we have $H_E(S)\geq\max(\HL(S),\HR(S).$  This fact, together with Corollary \ref{cor:H}, Proposition \ref{prop:semisimple} and Lemma \ref{lem:E}, yields:

\begin{prop}\label{prop:reg}
If $S$ is a regular semigroup, then 
$$\HL(S)=\HR(S)=\HH(S)=H_E(S)\geq\HJ(S).$$
\end{prop}

The following example shows that the inequality in Proposition \ref{prop:reg} can be strict, even for inverse semigroups.

\begin{ex}
The {\em bicyclic monoid}, denoted by $B,$ is the monoid defined by the presentation $\langle a,b\,|\,ab=1\rangle.$  The bicyclic monoid is inverse and simple (so $\HJ(B)=1$), but its set $E$ of idempotents forms an infinite chain $1>ba>b^2a^2>\cdots$ \cite[Theorem 2.53]{Clifford:1961}, so $H_E(B)\,(=\HL(B)=\HR(B)=\HH(B))$ is infinite.
\end{ex}


We now collect some equivalent characterisations for a semigroup to be both regular and stable, and deduce that for such a semigroup all the $\k$-heights coincide.  A semigroup is said to be {\em completely semisimple} if each of its principal factors is either completely simple or completely 0-simple.

\begin{prop}\label{prop:reg,stable}
For a semigroup $S,$ the following are equivalent:
\begin{enumerate}
\item $S$ is regular and stable;
\item $S$ is regular and either left stable or right stable;
\item $S$ is completely semisimple;
\item $S$ is semisimple and stable;
\item $S$ is regular and does not contain a copy of the bicyclic monoid.
\end{enumerate}
Moreover, if any (and hence all) of the conditions (1)-(5) hold, then 
$$\HL(S)=\HR(S)=\HH(S)=H_E(S)=\HJ(S).$$
\end{prop}

\begin{proof}
The final part of the statement follows from Propositions \ref{prop:stable}(3) and \ref{prop:reg}, so we only need to verify that (1)-(5) are equivalent.

First we remark that (2), (3) and (4) are equivalent by \cite[Theorems 6.45 and 6.48]{Clifford:1967}, and that (1)$\Rightarrow$(5) follows from \cite[Corollary 2.2]{Anderson}. 

Next we note that (1)$\Rightarrow$(2) is obvious, and that (3)$\Rightarrow$(1) follows from the fact that completely (0-)simple semigroups are regular and stable.
It remains to show that (5)$\Rightarrow$(3), so suppose that (5) holds.  Then each principal factor of $S$ contains an idempotent (since $S$ is regular) but no copy of the bicyclic monoid, so is either completely simple or completely 0-simple by \cite[Theorem 2.54]{Clifford:1961}.  Thus (3) holds, and the proof is complete.
\end{proof}




By Lemma \ref{lem:stable} and Proposition \ref{prop:reg,stable}, we have:

\begin{cor}\label{cor:reg}
If $S$ is a regular semigroup with finite $\j$-height, then
$$\HL(S)=\HR(S)=\HH(S)=H_E(S)=\HJ(S)\,\Longleftrightarrow\,S\text{ is stable}.$$
\end{cor}

\vspace{0.5em}
\section*{Acknowledgements}
This work was supported by the Engineering and Physical Sciences Research Council [EP/V002953/1, EP/V003224/1].  The authors thank the referee for helpful comments.

\end{document}